\documentclass[12pt]{article}

\usepackage{graphicx,amscd,amsmath,amssymb,verbatim,amsthm}
\usepackage[TS1,OT1]{fontenc}

\usepackage{color}
\definecolor{darkblue}{RGB}{70,130,180}

\def\R{\mathbb R}

\def\Z{\mathbb Z}
\def\N{\mathbb N}

\def\sgn{\mathrm{sgn\ }}

\def\<{\langle}
\def\>{\rangle}
\def\MOP{multiple orthogonal polynomials}

  \newtheorem{te}{Theorem}[section]

\begin{document}

\title{Interlacing properties of zeros of multiple orthogonal polynomials}
\author{Maciej Haneczok and Walter Van Assche}

\date{\today }
\maketitle
\begin{abstract}
It is well known that the zeros of orthogonal polynomials interlace. In this paper we study the case of multiple orthogonal polynomials. We recall known results and some recursion relations for \MOP. Our main result gives a sufficient condition, based on the coefficients in the recurrence relations,  for the interlacing of the zeros of neighboring \MOP. We give several examples illustrating our result. 
\end{abstract}

\section{Preliminaries}

An important and useful relation between zeros of consecutive polynomials is the following separation of zeros
\begin{eqnarray} 
y_1 < x_1 < y_2 <  \ldots  < x_n < y_{n+1}
\end{eqnarray}
where $x_1, \ldots , x_n$ and $y_1,\ldots , y_{n+1}$ are the zeros of the polynomials $p_n$ and $p_{n+1}$ respectively. The most important consequence of this property is that it simplifies estimates of the ratio of polynomials, needed for proving the asymptotic behavior. Among other things, the interlacing property of zeros guarantees convergence of the approximation of zeros by zeros of certain special functions using a fixed point method (see \cite{Segura}) and it is equivalent to the positivity of weights in quadrature formulas (e.g., \cite{Monegato}).

The zeros of orthogonal polynomials interlace as a consequence of the Christoffel-Darboux formula or the recurrence relation (e.g., \cite{Szego}). It holds even more general for a Sturm sequence of polynomials. We are interested in finding the interlacing property for multiple orthogonal polynomials from their recurrence relations.

In the first subsections \ref{subsection_mop} --  \ref{subsection_nnrr} we briefly introduce multiple orthogonal polynomials and their recursion. In Section \ref{section_interlacing_properties} we summarize known results, which are for the case of an AT system, a Nikishin system and a mixed type system with two Nikishin systems. We present our main result in Subsection \ref{subsection_iterlacing_properties_from_rr}, which states that the positivity of the coefficients $a_{\vec n,j}$ in the nearest neighbor recurrence relations ensures the interlacing of zeros of neighboring \MOP. It is worth noting that this condition does not depend directly on whether the measures $(\mu_1, \ldots , \mu _r)$, with respect to which the polynomials satisfy orthogonality conditions, form a certain special system. Section \ref{section_examples} contains several examples of our result. It gives examples of multiple Hermite polynomials (see \ref{subsection_hermite}), multiple Charlier polynomials (see \ref{subsection_charlier}), both kinds of multiple Meixner polynomials (see \ref{subsection_meixner_fk} and \ref{subsection_meixner_sk}), multiple Krawtchouk polynomials (see \ref{subsection_krawtchouk}), and both kinds of multiple Laguerre polynomials (see \ref{subsection_laguerre_sk} and \ref{subsection_laguerre_fk}).

In the following we will use the standard notation, i.e., by $\vec n$ we denote the multi-index $(n_1,\ldots , n_r)$, where $n_j\in \N$. By $|\vec n|$ we denote the length of the multi-index, i.e., $|\vec n|=\sum_{i=1}^r n_i$.

\subsection{Multiple orthogonal polynomials}\label{subsection_mop}

As is well known, there are two types of multiple orthogonal polynomials, type I and type II (e.g., \cite{Ismail,VanAssche}). In this paper we will mainly investigate \MOP\ of type II, namely the unique monic polynomial of degree $|\vec n|$ satisfying the orthogonality conditions 
\begin{eqnarray} \label{orthogonality_typeII}
\int _{\R} x^k P_{\vec n} (x)\, d\mu _j (x)  = 0,\ \quad \text{for}\ 0 \leq k\leq n_j-1,\ \ \text{and}\ \ j=1,\ldots r,
\end{eqnarray}
with respect to $r$ different positive Borel measures $(\mu _1,\ldots , \mu _r)$ that are absolutely continuous with respect to a measure $d\mu$, i.e., for all $j=1,\ldots r$ we have $d\mu_j (x) = w_j(x)\,d\mu(x)$. We say that the multi-index $\vec{n}$ is normal when this monic polynomial $P_{\vec{n}}$ is unique.

Type I multiple orthogonal polynomials are polynomials $(A_{\vec n,1}, \ldots, A_{\vec n,r})$, where $A_{\vec n,j}$ has degree $\leq n_j-1$ such that 
\begin{eqnarray}
\sum_{j=1}^r \int _{\R} x^k A_{\vec n,j} (x)\, d\mu _j (x)  = 0,\ \quad \text{for}\ 0 \leq k\leq |\vec n|-2,\ \ \text{and}\ \ j=1,\ldots r,
\end{eqnarray}
and
\begin{eqnarray}
\sum_{j=1} ^r \int _{\R} x^{|\vec n|-1} A_{\vec n,j} (x)\, d\mu _j (x)  = 1,\ \quad \text{for}\ j=1,\ldots r.
\end{eqnarray}
These polynomials are uniquely defined if and only if the multi-index $\vec n$ is normal. We will use the notation 
$Q_{\vec n} (x) = \sum_{j=1}^r A_{\vec n,j}(x) w_j(x)$.

Multiple orthogonal polynomials are intimately related to Hermite-Pad\'e approximants and often they are also called Hermite-Pad\'e polynomials.

\subsection{Recurrence relations for multiple orthogonal polynomials} \label{subsection_rrMOP}

Recall that orthogonal polynomials satisfy a three term recurrence relation 
\begin{eqnarray*}
x p_n(x) = a_{n+1} p_{n+1}(x) + b_n p_n(x) + a_n p_{n-1}(x) ,\ \qquad  n\geq 0,
\end{eqnarray*}
with initial values $p_0=1$ and $p_{-1}=0$ (e.g., \cite{Szego}). 
Multiple orthogonal polynomials also satisfy finite order recurrence relations. There are two types of recurrence relations for multiple orthogonal polynomials. Often only the multiple orthogonal polynomials are considered with the following sequences of indices \\

\noindent \begin{tabular}{p{0.7in}p{2.3in}p{2in}}
 diagonal: & $\vec n = (n,\ldots ,n)$, & $|\vec n|= r n$, \\
 step-line: & $\vec n = (m+1, \ldots , m+1, m, \ldots ,m)$, & $|\vec n| =rm + s,  0 \leq s \leq r-1$.
\vskip.6cm			
\end{tabular}

Any multi-index $\vec n$ on the step-line may be identified by the value of the length of $\vec n$. Observe that every natural number $n$ may be written as $n=m r+s$ with $0 \leq s\leq r-1$ and then the corresponding multi-index $\vec n$ is the one above where the value $m+1$ is repeated $s$ times. 
If all multi-indices are normal, then as a consequence of the orthogonality conditions we have the following finite order recurrence relation 
for the polynomials on the step-line:
\begin{eqnarray*}
x P_n(x) &=& P_{n+1}(x) + a_{n,n} P_n(x) + a_{n,n-1} P_{n-1}(x) + a_{n,n-2}P_{n-2}(x) \\
& & + \ldots + a_{n,n-r} P_{n-r} (x), \qquad n\geq 0,
\end{eqnarray*}
with initial conditions $P_{-r}=\ldots =P_{-1}=0$ and $P_0=1$.

There are several ways to increase the degree of multiple orthogonal polynomials, since we are working with multi-indices. Another recurrence relation only uses the nearest neigbors of $P_{\vec{n}}$. This will be described in the next subsection.

\subsection{Nearest neighbor recurrence relations} \label{subsection_nnrr}

Here we briefly introduce a second recursion for \MOP. For more details we refer to  \cite{VanAssche}. 
The nearest neighbor recurrence relation \eqref{nnrr}  connects type II multiple orthogonal polynomial $P_{\vec n}$ with the polynomial of degree one higher $P_{\vec n+\vec e_k}$ and all the neighbors of degree one lower $P_{\vec n- \vec e_j}$ for $j=1,\ldots , r$. To derive it, we proceed as follows.
Since both $P_{\vec n}$ and $P_{\vec n+\vec e_k}$ are monic polynomials, the difference  
\begin{eqnarray*}
x P_{\vec n}(x) - P_{\vec n +\vec e_k} (x)
\end{eqnarray*}
is a polynomial of degree $\leq |\vec n|$. Choosing $b_{\vec n,k}$ appropriately we can also cancel the term containing $x^{|\vec n|}$. Hence
\begin{eqnarray*}
x P_{\vec n}(x) - P_{\vec n +\vec e_k} (x) - b_{\vec n,k} P_{\vec n} (x)
\end{eqnarray*}
is a polynomial of degree $\leq |\vec n|-1$. It is easy to see that this polynomial is orthogonal to all polynomials of degree $n_j-2$ with respect to 
$\mu_j$, for $j=1,\ldots r$.
Finally, the polynomials $P_{\vec n - \vec e_1},\ldots , P_{\vec n - \vec e_r}$ form a basis for the linear space of all polynomials of degree $\leq |\vec n| -1$ that satisfy the orthogonality conditions
\begin{eqnarray*}
\int _{\R} x^k P_{\vec n} (x)\, d\mu _j (x)  = 0,\ \quad \text{for}\ k\leq n_j-2,\ \ \text{and}\ \ j=1,\ldots r.
\end{eqnarray*}
Hence we can write the polynomial $x P_{\vec n}(x) - P_{\vec n +\vec e_k} (x) - b_{\vec n,k} P_{\vec n} (x)$ as a linear combination of the polynomials 
$P_{\vec n-\vec e_1}, \ldots, P_{\vec n-\vec e_r}$, with some coefficients $a_{\vec n,j}$. This forms the so-called nearest neighbor recurrence relation 
\begin{eqnarray}\label{nnrr}
x P_{\vec n}(x)  =  P_{\vec n +\vec e_k} (x) + b_{\vec n,k} P_{\vec n} (x) + \sum _{j=1} ^r a_{\vec n,j} P_{\vec n- \vec e_j} .
\end{eqnarray}
There are $r$ such relations for each $k=1,\ldots ,r$. Note that the coefficients $a_{\vec{n},j}$ do not depend on $k$. Indeed, 
the coefficients $a_{\vec n,j}$ can be computed directly from this recursion using the orthogonality conditions (\ref{orthogonality_typeII}), giving
\begin{eqnarray}\label{anj}
a_{\vec n,j} = \frac{ \int _{\R} x^{n_j} P_{\vec n} (x)\, d\mu _j(x) } { \int _{\R} x^{n_j-1} P_{\vec n-\vec e_j} (x)\, d\mu _j(x) },
\end{eqnarray}
while the coefficients $b_{\vec n,k}$ can be computed by multiplying both sides of the equation by the multiple orthogonal polynomials of type I and applying the bi-orthogonality (see \cite{Ismail} Chapter 23, Theorem 23.1.6), giving 
\begin{eqnarray}\label{bnj}
b_{\vec n,k} = \int _{\R} x P_{\vec n} (x)  Q _{\vec n +\vec e_k} (x)\, d\mu(x). 
\end{eqnarray}

Similar recursion relations hold for type I \MOP, see \cite{Ismail,VanAssche}.

\section{Interlacing properties} \label{section_interlacing_properties}

D. Kershaw proved in \cite{Kershaw} the interlacing property for the zeros of polynomials orthogonal with respect to a Markov system. If the sequence $\phi _1(x), \phi _2(x), \ldots $ forms an integrable Markov system on $(a,b)$ (in particular a Chebyshev system) then there exists a unique polynomial $q_n$ of degree equal to $n$ with real, simple zeros in $[a,b]$, such that 
\begin{eqnarray*}
\int _a ^b q_n(x) \phi _i (x)\, dx = 0,\ \quad \text{for}\  i=1,2.\ldots , n,
\end{eqnarray*}
and if $q_{n+1}$ is a polynomial of degree exactly $n+1$ with real, simple zeros satisfying the same orthogonality condition, 
then the zeros of $q_n$ and $q_{n+1}$ interlace. 

Later, in \cite{FidalgoPrieto}, this was generalized in the sense that the Lebesgue measure was replaced by an arbitrary Borel measure, which was 
needed to prove interlacing in the case of \MOP\ for a Nikishin system. The authors showed that if $N(\sigma_1,\ldots \sigma_r)$ is an arbitrary Nikishin system of $r$ measures, then the zeros of the type II \MOP\ $P_{\vec n}$ and $P_{\vec n+\vec e_k}$ interlace and the zeros of $P_{\vec n}$ and $Q_{\vec n,j}$ interlace, where
\begin{eqnarray*}
Q_{\vec n,j} (x) = \int \frac{P_{\vec n}(x)-P_{\vec n} (t) }{x-t}\, ds_j(t),
\end{eqnarray*}
and $s_j$ (for $1\leq j\leq r$) are measures defined by induction using 
\begin{eqnarray} \label{sigma_induction}
\< \sigma_i,\sigma_j\> (x) = \int \frac{d \sigma _j (t)}{x-t}\, d\sigma _i(x)\ \quad \< \sigma _1, \sigma _2, \ldots , \sigma _n\> = \< \sigma _1, \< \sigma _2,\ldots , \sigma _n\> \> ,
\end{eqnarray}
i.e., $s_1=\< \sigma_1 \> = \sigma _1$, $s_2 =\< \sigma _1, \sigma _2 \>, \ldots , s_r =\< \sigma _1, \ldots , \sigma_r\>$.


In \cite{Aptekarev} the authors generalized the above result, proving that the zeros of $P_{\vec n, k}$ and $P_{\vec n+\vec e_\ell, k}$ interlace, 
where the sequence of polynomials $P_{\vec n,k}$ is defined as the sequence of polynomials of degree $n_k+\ldots +n_r$ with $P_{\vec n, 0}=P_{\vec n, r+1}\equiv 1$, that is orthogonal to the measures which are given by
\begin{eqnarray*}
d\mu_k (x) = \frac{|H_{\vec n, k}(x)|}{ |P_{\vec n, k-1}(x)P_{\vec n, k+1}(x)| } \, d\sigma_k (x),
\end{eqnarray*}
where $H_{\vec n,k}=\frac{P_{\vec n, k-1}\Psi_{\vec n, k-1}}{P_{\vec n, k}}$ and $\Psi _{\vec n,k}$ are defined recursively as follows, for each $\vec n$ we set $\Psi _{\vec n,0} (x) = P_{\vec n} (x)$ and
\begin{eqnarray*}
\Psi _{\vec n, k} (x) = \int \frac{\Psi _{\vec n, k-1} (t) }{x-t}\, d\sigma _k (t), \qquad \text{for}\ k =1, \ldots, r.
\end{eqnarray*}
In this notation the case $k=1$ corresponds to result in \cite{FidalgoPrieto}, i.e., the type II multiple orthogonal polynomials $P_{\vec n}$ are 
$P_{\vec n, 1}$.

Using arguments as in \cite[Property (P)]{Kershaw}, one can obtain the interlacing property for the type II \MOP\ with respect to measures that form an AT system. Recall that a system of measures $(\mu _1,\ldots , \mu_r)$ forms an AT system on $[a,b]$ if the measures $\mu _j$ are absolutely continuous with respect to a measure $\mu$ on $[a,b]$, with $d\mu_j(x) = w_j(x)\, d\mu(x)$, and
\begin{eqnarray*}
\{ w_1, xw_1,\ldots , x^{n_1-1}w_1, w_2\ldots , x^{n_r-1}w_r \}
\end{eqnarray*}
is a Chebyshev system on $[a,b]$. The zeros of type II \MOP\ for a AT system are real and simple (e.g., \cite{Ismail}).
The results in \cite{Kershaw} and \cite{FidalgoPrieto} can then be summarized as
\begin{te}\label{thmAT}
Suppose that the measures $(\mu _1,\ldots , \mu_r)$ form an AT system on $[a,b]$. Then the zeros of type II \MOP\ $P_{\vec n}$ and 
$P_{\vec n + \vec e_k}$ interlace.
\end{te}
\begin{proof}
Let $A$ and $B$ be two constants such that $|A|+|B|\not= 0$. Consider $AP_{\vec n} +BP_{\vec n + \vec e_k}$. 
Suppose that this polynomial has a multiple real zero at $c \in \mathbb{R}$, i.e., at least a double zero
\begin{eqnarray*}
AP_{\vec n}(x) +BP_{\vec n + \vec e_k}(x) = (x-c)^2 Q(x),
\end{eqnarray*}
for some polynomial $Q$ of degree $|n|-1$. Denote
\begin{multline*}
W_n (x_1,\ldots , x_{|\vec n|}) \\ 
= \det \left[ \begin{matrix} w_1(x_1) & \ldots & x_1^{n_1-1}w_1(x_1) & w_2(x_1) & \ldots & x_1 ^{n_r-1}w_r(x_1) \\
w_1(x_2) & \ldots &  x_2^{n_1-1} w_1 (x_2) & w_2(x_2) &\ldots & x_2^{n_r-1}w_r(x_2) \\
\vdots & \cdots& \vdots &\vdots & \cdots &\vdots \\
w_1(x_{|\vec n|}) & \ldots & x_{|\vec n|}^{n_1-1}w_1(x_{|\vec n|}) & w_2(x_{|\vec{n}|}) & \ldots & x_{|\vec n|}^{n_r-1} w_r(x_{|\vec n|})\end{matrix}\right]
\end{multline*}
for points $x_1,\ldots ,x_{|\vec n|}$ in $[a,b]$. This determinant vanishes if and only if $x_i=x_j$ for some $i,j$. 
We consider $W_n(x, x_1, \ldots x_{|\vec n|-1})$ as a function of $x$, and it changes sign only if $x$ passes through the points $x_1,\ldots , x_{|\vec n|-1}$.
If $Q$ is a polynomial of degree $|\vec n|-1$ with real coefficients, having precisely $r$ real zeros $x_1,\ldots x_r$ on $[a,b]$ at which $Q$ changes sign then $Q(x) W_{r+1}(x, x_1, \ldots x_r)$
does not change sign on $[a,b]$ (this is a Property (P) in \cite{Kershaw}). In particular
\[  \int_a^b  (x-c)^2 Q(x) W_{r+1}(x,x_1,\ldots , x_r) \, d\mu (x) \neq 0. \]
Using the orthogonality conditions  
\begin{eqnarray*}
\int _a ^b x^k (x-c)^2 Q(x) w_j(x)\, d\mu (x) = 0,\  \quad k=0,\ldots n_j-1,\ \ j=1,\ldots , r,
\end{eqnarray*}
we get
\begin{eqnarray*}
\int _a ^b (x-c)^2 Q(x) W_{r+1}(x,x_1,\ldots , x_r) \, d\mu (x) = 0,
\end{eqnarray*}
which forces a contradiction. This proves that $AP_{\vec n} +BP_{\vec n + \vec e_k}$ has only simple zeros on the real line. 
Hence the linear system of equations
\begin{eqnarray*}
\left[ \begin{matrix} P_{\vec n}(x) & P_{\vec n+\vec e_k}(x) \\ P'_{\vec n}(x) & P'_{\vec n+\vec e_k}(x) \end{matrix} \right]  \left[ \begin{matrix} A \\ B \end{matrix} \right] = \left[ \begin{matrix} 0 \\ 0 \end{matrix} \right] 
\end{eqnarray*}
has only the trivial solution $A=B=0$ for every $x\in \mathbb{R}$. Therefore the matrix on the left hand side has non-zero determinant and by continuity and the behavior for large $x$ we conclude that for all $x\in \mathbb{R}$
\begin{eqnarray} \label{critical_inequality}
P_{\vec n}(x)P'_{\vec n+\vec e_k}(x)  - P'_{\vec n}(x) P_{\vec n+\vec e_k}(x) >0 .
\end{eqnarray}
If we now fix two consecutive zeros $x_k$, $x_{k+1}$ of $P_{\vec n}$, then $P'_{\vec n} (x_k)P'_{\vec n}(x_{k+1})<0$. From the inequality (\ref{critical_inequality}) we get
\begin{eqnarray*}
P'_{\vec n}(x_k) P_{\vec n + \vec e_k}(x_k)<0\ \ \text{and}\ \ P'_{\vec n} (x_{k+1}) P_{\vec n + \vec e_k} (x_{k+1})<0. 
\end{eqnarray*}
Hence 
\begin{eqnarray*}
P_{\vec n+\vec e_k} (x_k) P_{\vec n+\vec e_k} (x_{k+1}) <0,
\end{eqnarray*}
which means that there is at least one zero $y_{k}$ of $P_{\vec n+ \vec e_k}$ between $x_k$ and $x_{k+1}$. If $x_{|\vec n|}$ is the greatest zero of 
$P_{\vec n}$ then $P'_{\vec n}(x_{|\vec n|})>0$ and from inequality (\ref{critical_inequality}) we have $P_{\vec n+\vec e_k} (x_{|\vec n|})<0$. But since $P_{\vec n+\vec e_k}$ is a monic polynomial, the sign of $P_{\vec n+\vec e_k}(x)$ for large enough $x$ is positive, which means that there is at least one zero of $P_{\vec n+\vec e_k}$ on the right side of $x_{|\vec n|}$.
Similarly there is at least one zero $y_1$ of $P_{\vec n+ \vec e_k}$ on the left side of the least zero $x_1$ of $P_{\vec n}$. Since the zeros of both polynomials are real and simple we have
\begin{eqnarray*}
y_j<x_{j}<y_{j+1}\ \quad \text{for}\ j=1,\ldots ,|\vec n|. 
\end{eqnarray*} 
\end{proof}

There is an alternative definition of an AT system: instead of zeros one may consider sign changes. In this case instead of using Property (P) from 
\cite{Kershaw} in the crucial step, one uses the normality of the index $\vec n$, which holds because $(x-c)^2(\mu _1, \ldots , \mu _r )$ is an AT system. 
Then the polynomial $Q$ of degree $|\vec{n}|-1$ satisfies the orthogonality conditions of the multiple orthogonal polynomial with multi-index
$\vec{n}$ for the AT-system $(x-c)^2 (\mu_1,\ldots,\mu_r)$, which is not possible because of the normality.

In \cite{FidalgoPrieto2} the authors investigate a slightly more general notion of multiple orthogonal polynomials, i.e., the case of mixed type multiple orthogonal polynomials. They investigate them for two Nikishin systems. Let $N(\sigma_1^1,\ldots \sigma_{r_1}^1)$ and $N(\sigma_1^2,\ldots \sigma_{r_2}^2)$ be two Nikishin systems of $r_1$ and $r_2$ measures respectively, which come from the same basis measure $\sigma _0^1=\sigma _0^2$. 
The polynomials $a_{\textbf{n},0},\ldots, a_{\textbf{n},r_1}$, where $\textbf{n} = (\textbf{n}_1, \textbf{n}_2 )$, $\textbf{n} _i= ( n_{i,0}, n_{i,1}, \ldots, n_{i,r_i})\in \Z ^{r_i+1}$ for $i=1,2$, such that 
\begin{enumerate}
\item the degree of $a_{\textbf{n},j}\leq n_{1,j}-1$, $j=0, \ldots , r_1$,  not all identically equal to zero,
\item for $i=0,\ldots , r_2$ they satisfy
\begin{eqnarray*}
\int x^k \Bigg( a_{\textbf{n},0}(x) + \sum_{j=1}^{r_1} a_{\textbf{n},j} (x) \hat s_{1,j}^1 (x)\Bigg)\, d s_{0,k}^2(x) =0,\quad k=0, \ldots, n_{2,i}-1, 
\end{eqnarray*}
where $s^1 _{j,k}$ are measures (and $\hat s^1 _{j,k}$ its Cauchy transform) defined by induction using (\ref{sigma_induction}), i.e., 
$s_{0,0}=\< \sigma_0 \> = \sigma _0$ and $s_{j,k}=\< \sigma _j,\ldots ,\sigma _k \>$,
\end{enumerate}
are called mixed type multiple orthogonal polynomials. They show (\cite[Theorem~3.5]{FidalgoPrieto2}) that the zeros of 
\begin{eqnarray*}
\mathcal{A}_{\textbf{n} , j }=a_{\textbf{n},0} + \sum_{i=j}^{r_1} a_{\textbf{n},i}(x) \hat s_{j+1,i}^1(x)\ \ \text{and}\ \ \mathcal{A}_{\textbf{n}^\ell,j} 
\end{eqnarray*}
interlace, where $\textbf{n}^\ell=(\textbf{n}_1+\vec e_{\ell_1},\textbf{n}_2+\vec e_{\ell_2})$, $0\leq \ell_1\leq r_1$ and $0\leq \ell_2\leq r_2$. 
Note that some special cases of this theorem are
\begin{enumerate}
\item if $r_1=0$ then the zeros of $a_{\textbf{n},0}$ and $a_{\textbf{n}^\ell,0}$ interlace, i.e.,  
   in our notation, that the zeros of type II \MOP\ $P_{\vec n}$ and $P_{\vec n+\vec e_k}$ interlace,
\item if $r_2=0$ and $j=0$ then the zeros of
\begin{eqnarray*}
\sum_{i=0}^{r_1} a_{\textbf{n},i}(x) \hat s_{1,i}^1(x)\ \ \text{and}\ \ \sum_{i=0}^{r_1} a_{\textbf{n}^\ell,i}(x) \hat s_{1,i}^1(x)
\end{eqnarray*}
 interlace, i.e., in our notation, that the zeros of the type I \MOP\ $Q_{\vec n}$ and $Q_{\vec n + \vec e_k}$ interlace. 
\end{enumerate}

\subsection{Interlacing properties from the recurrence relation} \label{subsection_iterlacing_properties_from_rr}

We are interested in finding the interlacing property from the nearest neighbor recurrence relations. 
Note that  $a_{\vec n,j}\not = 0$ in (\ref{nnrr}) whenever all multi-indices are normal and $n_j>0$. Indeed these 
coefficients are defined as a ratio of two integrals (\ref{anj}). Both integrals are non-zero due to normality of $\vec n$ and $\vec n+\vec e_k$. If 
\begin{eqnarray*}
\int _{\R} x^{n_j} P_{\vec n } (x)\, d\mu _j(x) = 0 
\end{eqnarray*}
then $P_{\vec n}$ would satisfy the same orthogonality conditions as $P_{\vec n + \vec e_j}$, which shows that there exist a polynomial of degree $\leq |\vec n|$ satisfying the orthogonality conditions for a multi-index of length $|\vec n|+1$ which contradicts the normality of $\vec n+\vec e_k$.

We assume that $P_{\vec n}=0$  if $n_j<0$ for at least one $j$ and for $|\vec n|=0$ we assume that $P_{\vec n}=1$. Our main result in this paper is 

\begin{te} \label{thm} Suppose that the zeros of $P_{\vec n}$ are real and simple, and all multi-indices are normal. If for all $1\leq j\leq r$ 
and for all multi-indices $\vec{n}$ one has $a_{\vec n,j}>0$ whenever $n_j >0$, then the zeros of $P_{\vec n}$ and $P_{\vec n +\vec e_k}$ interlace for every $k$ with $1 \leq k \leq r$.
\end{te}

\begin{proof}
We will use a proof by induction on the length of $\vec n$. For $|\vec n|=1$ we have one zero $x_1$ of $P_{\vec n}$. Evaluating the nearest neighbor recurrence relation \eqref{nnrr} for type II \MOP\ at $x_1$, we get
\begin{equation*}
P_{\vec n + \vec e_k} (x_1)= - \sum_{j=1}^r a_{\vec n, j} P_{\vec n -\vec e_j} (x_1).
\end{equation*}
From our assumption we know that $P_{\vec n -\vec e_j} =1$ only for one $j=j_0$ and for the other $j$ it is zero. It follows that the sign of 
$P_{\vec n + \vec e_k}$ at $x_1$ is minus the sign of $a_{\vec n ,j_0}$, hence it is negative.
On the other hand, since $ P_{\vec n + \vec e_k}$ is monic and of degree $2$, the sign of $P_{\vec n + \vec e_k}$ for large enough and small enough $x$ is necessarily positive. This means that the zero of $P_{\vec n}$ is between two zeros of $P_{\vec n +\vec e_k}$.

Suppose that for $|\vec n |=m-1$ the zeros of $P_{\vec n }$ and $P_{\vec n +\vec e_k}$ interlace for all $k=1,\ldots ,r$. 
This means that $P_{\vec n + \vec e_k}$ at the zeros $x_i$ of $P_{\vec n}$ has alternating signs, i.e.,
\begin{equation*}
\sgn  P_{\vec n + \vec e_k} (x_i)  = \begin{cases} (-1)^{i+1}\ & \textrm{if } |\vec n|\ \text{is even} \\
                                                   (-1)^{i}\ & \textrm{if } |\vec n|\ \text{is odd} 
                                     \end{cases} 
\end{equation*}
Consider the case $|\vec n|=m$. Let $x_1,\ldots x_m$ be the zeros of $P_{\vec n}$. Evaluating the nearest neighbor recurrence relation 
\eqref{nnrr} at $x_i$ ($1\leq i\leq m$) we get
\begin{equation*}
P_{\vec n + \vec e_k} (x_i)= - \sum_{j=1}^r a_{\vec n, j} P_{\vec n -\vec e_j} (x_i).
\end{equation*} 
Denote by $y_1,\ldots y_{m+1}$ the zeros of $P_{\vec n + \vec e_k}$. From the induction assumption we know that the zeros of $P_{\vec n -\vec e_j}$ 
and $P_{\vec n}$ interlace, i.e., 
\begin{equation*}
\sgn P_{\vec n - \vec e_j} (x_i) = \begin{cases} (-1)^i\ & \textrm{if } |\vec n|\ \text{is even} \\
                                                 (-1)^{i+1}\ & \textrm{if } |\vec n|\ \text{is odd} \end{cases} .
\end{equation*}
Therefore, since $a_{\vec n, j}\geq 0$ and at least for one $j$ with $1 \leq j \leq r$ we have $a_{\vec{n},j} >0$ 
\begin{equation*}
\sgn P_{\vec n + \vec e_k} (x_i) = \begin{cases}(-1)^{i+1}\ & \textrm{if } |\vec n|\ \text{is even} \\
                                                (-1)^{i+2}\ & \textrm{if } |\vec n|\ \text{is odd} \end{cases} .
\end{equation*}
Both polynomials $P_{\vec n }$ and $P_{\vec n + \vec e_k}$ are monic and therefore for large enough $x$ their sign is positive. 
But at the point $x_m$ (the largest zero of $P_{\vec n}$), the polynomial $P_{\vec n + \vec e_k}$ has negative sign. Therefore there is at least one zero $y_{m+1}$ of $P_{\vec n + \vec e_k}$ such that $x_m<y_{m+1}$.
Similarly there is at least one zero of $P_{\vec n + \vec e_k}$ on the left hand side of the smallest zero of $P_{\vec n}$, i.e., $y_1<x_1$. 
Consequently, since the zeros of $P_{\vec n +\vec e_k}$ are real and simple, we have exactly one zero of $P_{\vec n}$ between two consecutive zeros of 
$P_{\vec n +\vec e_k}$ 
\begin{equation*}
y_j<x_{j}<y_{j+1}\ \quad \text{for}\ j=1,\ldots ,m. 
\end{equation*}
\end{proof}

\section{Examples}\label{section_examples}

Theorem \ref{thm} proves interlacing of zeros of many families of multiple orthogonal polynomials, such as multiple Hermite, Charlier, Meixner of the first kind, Krawtchouk, and Laguerre polynomials of the second kind. However, there are also examples where the positivity condition 
($a_{\vec n,j} > 0$ for all $\vec{n}$ with $n_j >0$) is too strong, showing that the condition is sufficient but not necessary. For instance in the case of multiple Laguerre polynomials of the first kind, the coefficients $a_{\vec n,j}$ are not all positive. A similar situation occurs in the case of multiple Meixner polynomials of the second kind, since in some sense Meixner polynomials are the discrete analogs of the Laguerre polynomials.  
Nevertheless both examples have the interlacing property, which follows from the fact that the measures, with respect to which these polynomials satisfy orthogonality conditions, form an AT system.

\subsection{Multiple Hermite polynomials} \label{subsection_hermite}

Multiple Hermite polynomials (see \cite{Ismail}) are given by the Rodrigues formula
\begin{eqnarray*}
H_{\vec n} ^{\vec c}(x) = (-1)^{|\vec n|} 2 ^{-|\vec n|} e^{x^2} \prod_{j=1}^r \big( e^{-c_jx} \frac{d^{n_j}}{dx^{n_j}} e^{c_jx} \big) e^{-x^2} ,
\end{eqnarray*}
where $c_1,\ldots , c_r$ are distinct real numbers. These polynomials are orthogonal with respect to measures $(\mu_1,\ldots ,\mu _r)$ which are given by $d\mu _j(x) = e^{-x^2+c_jx}\,dx$ on $(-\infty , \infty)$. These measures form an AT system. 
An explicit expression for multiple Hermite polynomials is
\begin{eqnarray*}
H_{\vec n} ^{\vec c}(x) &=& \frac{ (-1)^{|\vec n|} }{2 ^{|\vec n|}} \sum _{k_1=0}^{n_1}\ldots \sum _{k_r=0}^{n_r}{n_1\choose k_1} \ldots {n_r\choose k_r} \prod_{j=1}^r c_j^{n_j-k_j} (-1)^{|\vec k|} H_{|\vec k|} (x) ,
\end{eqnarray*}
where $H_{|\vec k|}$ is the usual Hermite polynomial.
From both formulas for $H_{\vec n} ^{\vec c}$ we get the coefficients in the nearest neighbor recurrence relation, and the recurrence relation is
\begin{eqnarray*}
x H_{\vec n} ^{\vec c}(x) = H_{\vec n + \vec e_k} ^{\vec c}(x) + \frac{c_k}{2} H_{\vec n} ^{\vec c}(x) + \sum_{j=1}^r \frac{n_j}{2} H_{\vec n-\vec e_j} ^{\vec c}(x) ,
\end{eqnarray*} 
so that $b_{\vec{n},k} = c_k/2$ and $a_{\vec{n},j} = n_j/2$. 
Clearly $a_{\vec n, j}>0$ whenever $n_j > 0$ and therefore Theorem \ref{thm} gives the interlacing property for the zeros of multiple Hermite polynomials.

\subsection{Multiple Charlier polynomials} \label{subsection_charlier}

Multiple Charlier polynomials (see \cite{Ismail}) are given by the Rodrigues formula
\begin{eqnarray*}
C_{\vec n} ^{\vec a} (x) = \prod _{j=1}^r (-a_j)^{n_j} \Gamma (x+1) \prod _{j=1}^r \big( a_j^{-x} \nabla  ^{n_j} a_j^x \big) \frac{1}{\Gamma (x+1)} ,
\end{eqnarray*}
where $a_1,\ldots , a_r>0$ and $a_i\not= a_j$ whenever $i\not=j$, and $\nabla$ is the backward difference operator, i.e.,
\begin{equation*}
\nabla f(x) = f(x)-f(x-1) .
\end{equation*}
These polynomials are orthogonal with respect to discrete measures $(\mu_1,\ldots , \mu_r)$ which are given by Poisson measures 
\begin{eqnarray*}
\mu_j=\sum_{k=0}^{\infty} \frac{a_j^k}{k!} \delta _k. 
\end{eqnarray*}
These measures form an AT system.
An explicit expression for multiple Charlier polynomials is
\begin{eqnarray*}
C_{\vec n} ^{\vec a} (x) &=& \sum_{k_1=0}^{n_1}\ldots \sum_{k_r=0}^{n_r} {n_1\choose k_1}\ldots {n_r\choose k_r} \prod _{j=1}^r (- a_j)^{n_j-k_j} (-1)^{|\vec k|}(-x)_{|\vec k|} . 
\end{eqnarray*}
From both formulas for $C_{\vec n} ^{\vec a}$ we get the coefficients in the nearest neighbor recurrence relation, and the recurrence relation is
\begin{eqnarray*}
x C_{\vec n} ^{\vec a}(x) = C_{\vec n + \vec e_k} ^{\vec a}(x) + (a_k + |\vec n|) C_{\vec n} ^{\vec a}(x) + \sum_{j=1}^r a_jn_j  C_{\vec n-\vec e_j} ^{\vec a}(x) ,
\end{eqnarray*}
so that $b_{\vec{n},k} = a_k + |\vec{n}|$ and $a_{\vec{n},j} = a_j n_j$. 
Since $a_j>0$ we have $a_{\vec n,j} >0$ whenever $n_j >0$ and therefore by Theorem \ref{thm} we get the interlacing property for the zeros of multiple Charlier polynomials.

\subsection{Multiple Meixner polynomials of the first kind} \label{subsection_meixner_fk}

Multiple Meixner polynomials of the first kind (see \cite{Ismail}) are given by the Rodrigues formula
\begin{eqnarray*}
M_{\vec n} ^{\beta , \vec c} (x) &=& (\beta )_{|\vec n|} \prod _{j=1}^r \Big( \frac{c_j}{c_j-1} \Big) ^{n_j} \frac{\Gamma (\beta ) \Gamma (x+1)}{ \Gamma (\beta +x)} \prod _{j=1}^r \Big( c_j^{-x} \nabla ^{n_j} c_j^x \Big) \\
& & \times\  \frac{\Gamma (|\vec n|+\beta +x)}{\Gamma (|\vec n|+\beta )\Gamma (x+1)} ,
\end{eqnarray*}
where $\beta >0$, $0 < c_i < 1$ for all $i$ and $c_i\not= c_j$ whenever $i\not= j$. These polynomials are orthogonal with respect to discrete measures $(\mu_1,\ldots ,\mu _r)$ which are given by 
\begin{eqnarray*}
\mu _j = \sum _{k=0}^{\infty} \frac{(\beta )_k c_j^k}{k!}\delta _k .
\end{eqnarray*}
These measures form an AT system. 
We can compute an explicit expression for multiple Meixner polynomials of the first kind 
\begin{eqnarray*}
M_{\vec n} ^{\beta , \vec c} (x) = \sum_{k_1=0}^{n_1}\ldots \sum_{k_r=0}^{n_r} {n_1\choose k_1}\ldots {n_r\choose k_r}  \prod _{j=1}^r \frac{c_j^{n_j-k_j}}{(c_j-1) ^{n_j}}  (-x)_{|\vec k|} (\beta +x)_{|\vec n| -|\vec k|} .
\end{eqnarray*}
From both formulas for $M_{\vec n} ^{\beta , \vec c}$ we can compute the coefficients in the nearest neighbor recurrence relation, and the recurrence relation is
\begin{eqnarray*}
x M_{\vec n} ^{\beta , \vec c}(x) &=& M_{\vec n + \vec e_k} ^{\beta , \vec c}(x) + \Bigg( (|\vec n|+\beta ) \frac{c_k}{1-c_k} + \sum_{i=1}^r \frac{n_i}{1-c_i}\Bigg)  M_{\vec n} ^{\beta , \vec c}(x) \\
&& +\ \sum_{j=1}^r  c_jn_j \frac{(\beta + |\vec n|-1) } { (1-c_j)^2 }  M_{\vec n-\vec e_j} ^{\beta , \vec c}(x) ,
\end{eqnarray*}
so that
\[   b_{\vec{n},k} = (|\vec n|+\beta ) \frac{c_k}{1-c_k} + \sum_{i=1}^r \frac{n_i}{1-c_i}, \quad
     a_{\vec{n},j} = c_jn_j \frac{(\beta + |\vec n|-1) } { (1-c_j)^2 }.  \]
These recurrence coefficients were not computed earlier and appear here for the first time.
Since $0 < c_j < 1$ and $\beta >0$ we see that $a_{\vec n, j}>0$ whenever $n_j>0$ and therefore by Theorem \ref{thm} we get the interlacing property of the zeros of multiple Meixner polynomials of the first kind.

\subsection{Multiple Krawtchouk polynomials} \label{subsection_krawtchouk}

Multiple Krawtchouk polynomials (see \cite{Ismail}) are type II multiple orthogonal polynomials that are orthogonal with respect to binomial measures
\begin{eqnarray*}
\mu _j = \sum _{k=0}^{ N} {N\choose k} p_i^k (1-p_i)^{N-k} \delta _k .
\end{eqnarray*}
For $|\vec n| \leq N$ they are multiple Meixner polynomials of the first kind with $\beta = -N$ and $c_i=\frac{p_i}{p_i-1}$, 
for $0 < p_i <1$ for all $i$. 
Hence in this case we can immediately write the nearest neighbor recurrence relation using \ref{subsection_meixner_fk}, i.e.,
\begin{eqnarray*}
x K_{\vec n} ^{\vec p, N} (x)  &=& K_{\vec n +\vec e_k} ^{\vec p, N} (x)  + \Bigg( (N-|\vec n|) p_k +\sum_{i=1}^r n_i(1-p_i)\Bigg)  K_{\vec n} ^{\vec p, N} (x) \\
&& +\ \sum_{j=1}^r  \frac{p_j}{p_j-1}n_j \frac{(|\vec n|-N-1) } { (\frac{-1}{p_j-1})^{2} }  K_{\vec n -\vec e_j} ^{\vec p, N} (x)  ,
\end{eqnarray*}
so that
\[    b_{\vec{n},k} = (N-|\vec n|) p_k +\sum_{i=1}^r n_i(1-p_i), \quad
      a_{\vec{n},j} = \frac{p_j}{p_j-1}n_j \frac{(|\vec n|-N-1) } { (\frac{-1}{p_j-1})^{2} }. \]
Since $0 < p_j < 1$ and $N>0$ we see that $a_{\vec n, j}>0$ whenever $n_j > 0$ and therefore Theorem \ref{thm} gives the interlacing property of the zeros of the Multiple Krawtchouk polynomials for which $|\vec{n}|\leq N$.

\subsection{Multiple Laguerre polynomials of the second kind} \label{subsection_laguerre_sk}

Multiple Laguerre polynomials of the second kind (see \cite{Ismail}) are given by the Rodrigues formula
\begin{eqnarray*}
L_{\vec n} ^{\vec c}(x) = (-1)^{|\vec n|} \prod _{i=1}^r c_i^{-n_i}x^{-\alpha} \prod_{j=1}^r \big( e^{c_jx} \frac{d^{n_j}}{dx^{n_j}} e^{-c_jx} \big) x^{|\vec n |+\alpha } ,
\end{eqnarray*}
where $\alpha >-1$, $c_j>0$ and $c_i\not=c_j$ whenever $i\not=j$. These polynomials are orthogonal with respect to measures $(\mu_1,\ldots ,\mu _r)$ which are given by $d\mu _j(x) = x^{\alpha}e^{-c_jx}\,dx$ on $[0, \infty)$. These measures form an AT system.
One can compute an explicit expression for the multiple Laguerre polynomials of the second kind using Leibniz' rule:
\begin{eqnarray*}
L_{\vec n} ^{\vec c}(x) &=& \sum _{k_1=0}^{n_1}\ldots \sum _{k_r=0}^{n_r}{n_1\choose k_1} \ldots {n_r\choose k_r} {{|\vec n|+\alpha}\choose {|\vec k|} }|\vec k|!  \frac{ (-1)^{|\vec k|}} {\prod _{j=1}^r c_j^{k_j}} x^{|\vec n| -|\vec k|} .
\end{eqnarray*}
From both formulas for $L_{\vec n} ^{\vec c}$ we get the coefficients in the nearest neighbor recurrence relation, and the recurrence relation is
\begin{eqnarray*}
x L_{\vec n} ^{\vec c}(x) &=& L_{\vec n + \vec e_k} ^{\vec c}(x) + \left( \frac{(|\vec n|+1+\alpha ) }{c_k} + \sum _{j=1}^r \frac{n_j}{c_j} \right) 
L_{\vec n} ^{\vec c}(x) \\
&& +\ \sum_{j=1}^r \frac{n_j}{c_j^2 } (|\vec n|+\alpha)\ L_{\vec n-\vec e_j} ^{\vec c}(x) ,
\end{eqnarray*}
so that
\[   b_{\vec{n},k} = \frac{(|\vec n|+1+\alpha ) }{c_k} + \sum _{j=1}^r \frac{n_j}{c_j} , \quad
     a_{\vec{n},j} = \frac{n_j}{c_j^2 } (|\vec n|+\alpha). \]
Since $\alpha >-1$ we see that $a_{\vec n, j}>0$ whenever $n_j > 0$ and from Theorem \ref{thm} we then get the interlacing property for the zeros of multiple Laguerre of the second kind.

\subsection{Multiple Laguerre polynomials of the first kind} \label{subsection_laguerre_fk}

Multiple Laguerre polynomials of the first kind (see \cite{Ismail}) are given by the Rodrigues formula
\begin{eqnarray*}
L_{\vec n} ^{\vec \alpha}(x) = (-1)^{|\vec n|}e^x \prod_{j=1}^r \big( x^{-\alpha _j} \frac{d^{n_j}}{dx^{n_j}} x^{n_j+\alpha _j}\big) e^{-x},
\end{eqnarray*}
where $\alpha _i - \alpha _j\notin \Z$ whenever $i\not= j$, $\alpha _j >-1$. These polynomials are orthogonal with respect to measures $(\mu _1, \ldots , \mu _r)$ which are given by $d\mu _j (x) = x^{\alpha _j} e^{-x}\,dx$ on $[0,\infty )$. These measures form an AT system.
Again one can compute an explicit expression for multiple Laguerre polynomials of the first kind using Leibniz' rule:
\begin{eqnarray*}
L_{\vec n} ^{\vec \alpha}(x) &=& \sum_{k_1=0}^{n_1}\ldots \sum_{k_r=0}^{n_r} {n_1\choose k_1} \ldots {n_r\choose k_r} {{n_r+\alpha _r}\choose {k_r}}\ldots \\
&& \times\ {{|\vec n|- |\vec k| +k_1 +\alpha _1}\choose {k_1} }\prod _{i=1} ^r k_i! (-1)^{|\vec k|} x^{|\vec n|- |\vec k|} .
\end{eqnarray*}
From both formulas for $L_{\vec n} ^{\vec \alpha}$ we get the coefficients in the nearest neighbor recurrence relation, and the recurrence relation is
\begin{eqnarray*}
x L_{\vec n} ^{\vec \alpha} (x) &=& L_{\vec n + \vec e_k} ^{\vec \alpha} (x) + (|\vec n| + 1 + n_k + \alpha _k) L_{\vec n} ^{\vec \alpha} (x) \\
&& +\ \sum_{j=1}^r  n_j(n_j+\alpha _j) \prod _{i\not= j} ^r \frac{  \alpha _i -n_j - \alpha _j }{ n_i+\alpha _i -n_j-\alpha _j} \ L_{\vec n - \vec e_j} ^{\vec \alpha} (x) ,
\end{eqnarray*}
so that
\[   b_{\vec{n},k} = |\vec n| + 1 + n_k + \alpha _k ,  \quad
     a_{\vec{n},j} = n_j(n_j+\alpha _j) \prod _{i\not= j} ^r \frac{  \alpha _i -n_j - \alpha _j }{ n_i+\alpha _i -n_j-\alpha _j}. \]
Observe that $a_{\vec n, j}\not> 0$ and therefore we cannot apply Theorem \ref{thm}, but these polynomials have the interlacing property since the measures  $(\mu _1, \ldots , \mu _r)$ form an AT system so that Theorem \ref{thmAT} can be applied. It is worth noting that in \cite{VanAssche} it was shown that $\sum _{j=1}^r a_{\vec n,j} >0$. It is not clear whether such a condition is sufficient to prove the interlacing property.

\subsection{Multiple Mexiner polynomials of the second kind}\label{subsection_meixner_sk}

Multiple Mexiner polynomials of the second kind (see \cite{Ismail}) are given by the Rodrigues formula
\begin{eqnarray*}
M_{\vec n} ^{\vec \beta , c} (x) &=& \Big( \frac{c}{c-1} \Big) ^{|\vec n|} \prod_{j=1}^r (\beta _j)_{n_j} \frac{\Gamma (x+1)}{ c^x } \prod _{j=1}^r \Big( \frac{\Gamma (\beta _j)}{\Gamma (\beta _j +x)} \\
&&  \times\ \nabla ^{n_j} \frac{\Gamma (\beta _j + n_j + x)}{\Gamma (\beta _j + n_j )} \Big) \frac{c^x}{\Gamma (x+1)} ,
\end{eqnarray*}
where $0 < c < 1$ and $\beta_i -\beta_j\notin \Z$ whenever $i\not= j$, $\beta _j >0$. These polynomials are orthogonal with respect to discrete measures 
$(\mu_1,\ldots ,\mu _r)$ which are given by 
\begin{eqnarray*}
\mu _j = \sum _{k=0}^{\infty} \frac{(\beta _j)_k c^k}{k!}\delta _k .
\end{eqnarray*}
These measures form an AT system.
An explicit expression for the multiple Meixner polynomials of the first kind is 
\begin{eqnarray*}
M_{\vec n} ^{\vec \beta , c} (x) &=& \sum_{k_1=0}^{n_1}\ldots \sum_{k_r=0}^{n_r} {n_1\choose k_1}\ldots {n_r\choose n_r} \frac{c^{|\vec n|-|\vec k|}}{(c-1)^{|\vec n|}} \\
& & \times\ \prod _{j=1}^r (\beta _j+x -\sum_{i=1}^{j-1}k_i)_{n_j-k_j} (-x)_{|\vec k|} .
\end{eqnarray*}
From both formulas for $M_{\vec n} ^{\vec \beta, c}$ we get the coefficients in the nearest neighbor recurrence relation, and the recurrence relation is
\begin{eqnarray*}
x M_{\vec n} ^{\vec \beta , c}(x) &=& M_{\vec n + \vec e_k} ^{\vec \beta , c}(x) + \Big( \frac{|\vec n|}{1-c} + ( n_k + \beta _k ) \frac{c}{1-c}  \Big) 
 M_{\vec n} ^{\vec \beta , c}(x) \\
&& +\ \sum_{j=1}^r cn_j \frac{\beta _j +n_j-1}{(1-c)^2} \prod _{i\not= j} ^r \frac{  \beta _i -n_j - \beta _j }{ n_i+\beta _i -n_j-\beta _j} \ 
  M_{\vec n-\vec e_j} ^{\vec \beta , c}(x) ,
\end{eqnarray*}
so that
\[   b_{\vec{n},k} = \frac{|\vec n|}{1-c} + ( n_k + \beta _k ) \frac{c}{1-c}, \]
\[   a_{\vec{n},j} = cn_j \frac{\beta _j +n_j-1}{(1-c)^2} \prod _{i\not= j} ^r \frac{  \beta _i -n_j - \beta _j }{ n_i+\beta _i -n_j-\beta _j}  . \]
These recurrence coefficients were not computed earlier and appear here for the first time.
Again we see that $a_{\vec n,j}\not>0$ for all $j$. However the zeros of multiple Meixner of the second kind do interlace because of Theorem \ref{thmAT}. For these \MOP\ one can also show that $\sum_{j=1}^r a_{\vec n,j}>0$.

\begin{verbatim}
Department of Mathematics
Katholieke Universiteit Leuven
Celestijnenlaan 200B box 2400
BE-3001 Leuven 
Belgium
maciej.haneczok@wis.kuleuven.be
walter@wis.kuleuven.be
\end{verbatim}

\end{document}